\documentclass{amsart}
\usepackage{amsmath, enumerate}
\usepackage{algorithm}
\usepackage{algpseudocode}
\usepackage{amsfonts}
\usepackage{amssymb}
\usepackage{amsthm}
\usepackage{color}
\usepackage{caption}
\usepackage{graphicx}
\usepackage{url}

\newcommand{\Z}{\mathbb{Z}}
\newcommand{\gam}{\Gamma}

\newcommand{\N}{\mathbb{N}}

\newcommand{\Q}{\mathbb{Q}}

\newcommand{\bp}{\begin{problem}}

\newcommand{\ep}{\end{problem}}

\newcommand{\ba}{\begin{answer}}

\newcommand{\ea}{\end{answer}}

\newcommand{\ben}{\renewcommand{\theenumi}{\alph{enumi}}

\renewcommand{\labelenumi}{(\theenumi)}\begin{enumerate}}

\newcommand{\een}{\end{enumerate}}

\newcommand{\Out}{\mathrm{Out}}

\newtheorem{defin}{Definition}[section]
\newtheorem{thm}[defin]{Theorem}
\newtheorem{cor}[defin]{Corollary}

\newtheorem{lem}[defin]{Lemma}

\newtheorem{prob}[defin]{Problem}

\DeclareCaptionType{result}[Table][List of Results]

\title[Colorability and RAAGs]{An algebraic characterization of $k$--colorability}
\begin{document}
\date{\today}
\author[R. Flores]{Ram\'{o}n Flores}
\address{Ram\'{o}n Flores, Department of Geometry and Topology, University of Seville, Spain}
\email{ramonjflores@us.es}

\author[D. Kahrobaei]{Delaram Kahrobaei}
\address{Delaram Kahrobaei, Department of Computer Science, University of York, UK, New York University, Tandon School of Engineering, PhD Program in Computer Science, CUNY Graduate Center}
\email{dk2572@nyu.edu, delaram.kahrobaei@york.ac.uk}

\author[T. Koberda]{Thomas Koberda}
\address{Thomas Koberda, Department of Mathematics, University of Virginia, Charlottesville, VA 22904}
\email{thomas.koberda@gmail.com}

\begin{abstract}
We characterize $k$--colorability of a simplicial graph via the intrinsic algebraic structure of the associated right-angled Artin group.  As
a consequence, we show that a certain problem about the existence of  homomorphisms from right-angled  Artin groups to products
of free groups is NP--complete.
\end{abstract}

\maketitle

\section{Introduction}
\label{Intro}

Let $\gam$ be a finite simplicial graph with vertex set $V=V(\gam)$ and edge set $E=E(\gam)$.
We say that $\gam$ is \emph{$k$--colorable} if there is a map $\kappa\colon V\to \{1,\ldots,k\}$ such that if  $\{v,w\}\in E$ then
$\kappa(v)\neq\kappa(w)$. The map $\kappa$ is called a \emph{$k$--coloring} of $\gam$.
 The minimal $k$ for which there exists such a map $\kappa$ is called the \emph{chromatic number} of
$\gam$ and is denoted $\chi(\gam)$.

The problem of finding a $k$--coloring of a given graph is fundamental in graph theory, and has many applications in discrete mathematics,
computational complexity, and computer science. Determining whether a given graph is $3$--colorable is known to be NP--complete, as is
determining a graph's chromatic number~\cite{GJ1979,Minsky67,AB09}.

As such, determining the existence of a $k$--coloring of a graph is a fundamental problem in theoretical complexity theory,
with a plethora of applications to
both theoretical and applied computer science; for instance, several cryptographic schemes based on the $3$--colorability problem
have been proposed, such as a post-quantum public key encryption known as Polly Cracker~\cite{LMPT}, and a zero-knowledge
proof system for graph
$3$--colorability~\cite{GMW91}.

This paper studies $k$--colorings of graphs via algebraic methods, specifically right-angled Artin groups~\cite{CharneyGD}.
To  a finite simplicial graph $\gam$,
we associate the right-angled Artin group $A(\gam)$ via \[A(\gam)=\langle V(\gam)\mid [v,w]=1 \textrm{ if and only if } \{v,w\}\in E(\gam)\rangle.\]

It is well-known that the isomorphism type of $A(\gam)$ determines the isomorphism type of $\gam$~\cite{Droms87,Sabalka09,KoberdaGAFA},
so that the graph theoretic properties of $\gam$ should be reflected in the algebra of $A(\gam)$. We are generally interested in the following
problem:
\begin{prob}\label{prob:property}
Let $P$ be a property of finite simplicial graphs. Find a property $Q$ of groups such that $\gam$ has $P$ if and only if $A(\gam)$ has
$Q$.
\end{prob}
In order for Problem~\ref{prob:property} to be interesting, one should insist that $Q$ be a property of the isomorphism type of a group only.
In particular, one should disallow reference to a generating set.

One reason for interest in Problem~\ref{prob:property} is that it can provide
insight into various graph theoretic problems and their computational complexity
from the more flexible point of view of the algebraic structure of $A(\gam)$. Moreover, algebraically formulated problems
can be approached using an arbitrary presentation  for $A(\gam)$, without reconstructing the full underlying graph of $A(\gam)$.

Here are some instances where a satisfactory answer to Problem~\ref{prob:property} is known:

\begin{enumerate}
\item
The graph $\gam$ is a nontrivial join if and only if $A(\gam)$ decomposes as a nontrivial direct product~\cite{Servatius1989}.
\item
The graph $\gam$ is disconnected if and only if $A(\gam)$ decomposes as a nontrivial free product~\cite{BradyMeier01}.
\item
The graph $\gam$ is square--free if and only if $A(\gam)$ does not contain a subgroup isomorphic to a product $F_2\times F_2$ of
nonabelian free groups~\cite{Kambites09,KK2013gt}.
\item 
The graph $\gam$ admits an independent set $D$ of vertices
such that every cycle in $\gam$ meets $D$ at least twice  if and only if the poly-free length of $A(\gam)$ is two~\cite{HS2007}.
\item
The graph $\gam$ is a cograph (i.e. a $P_4$--free graph) if and only if $A(\gam)$ is contained in the class of finitely generated groups which
contains $\Z$ and is closed under taking direct products and free products~\cite{KK2013gt,KK2018jt}.
\item 
The graph $\gam$ is a finite tree or a finite complete bipartite graph if and only if $A(\gam)$ is a semidirect product of
two free groups of finite rank~\cite{HS2007}.
\item
The graph $\gam$ admits a nontrivial automorphism if and only if the group $\Out(A(\gam))$ of outer automorphisms of $A(\gam)$ contains
a finite nonabelian group~\cite{FKK2019}.
\item
A sequence of graphs $\{\gam_i\}_{i\in\N}$ forms a graph expander family if and only if the cohomology rings $\{H^*(A(\gam_i),F)\}$ over an
arbitrary field $F$
 form a \emph{vector space expander family}~\cite{FKK2020exp}.
\end{enumerate}

In this paper, we prove the following result which characterizes the existence of a $k$--coloring of a graph $\gam$ (and in particular the
chromatic number of $\gam$) via right-angled Artin groups, thus providing an answer to Problem~\ref{prob:property} when $P$ is
$k$--colorability.

\begin{thm}\label{thm:main}
Let $\gam$ be a finite simplicial graph with $n$ vertices. The graph $\gam$ is $k$--colorable if and only if there is a surjective map
\[A(\gam)\to \prod_{i=1}^k F_i,\] where for $1\leq i\leq k$ the group $F_i$ is a free group of rank $m_i$, and where \[\sum_{i=1}^k m_i=n.\]
\end{thm}

We note that the number of vertices of $\gam$ is a canonical invariant of the isomorphism type of $A(\gam)$, since this is simply the rank
of the abelianization of $A(\gam)$. We state Theorem~\ref{thm:main} as we do
because  of the conciseness of  the hypotheses and the
conclusion,  though  in the course  of the  proof  it  will  become clear that we can weaken
the  hypotheses  somewhat. For  instance,  the target groups need  not be free,  but may be replaced by groups with the correct Betti numbers
in which infinite order elements have (virtually) cyclic centralizers. This last hypothesis is satisfied in all Gromov hyperbolic
groups~\cite{Gromov1987}, for instance. Moreover, we need not assume that $\phi$ is surjective ---  it  suffices  that $\phi$ induces an
isomorphism on  first rational homology.

Theorem~\ref{thm:main} has the following easy consequence, which  is of interest from the point of view of complexity theory.
\begin{cor}
Fix $k\geq 3$. The problem of determining whether a right-angled Artin group $A(\gam)$ surjects to a product of $k$ free groups, the sum
of whose ranks equals $|V(\gam)|$, is NP--complete.
\end{cor}
\begin{proof}
It is a direct consequence of Theorem~\ref{thm:main} and the fact that the $k$-colorability problem is NP--complete for $k\geq 3$ (see for example \cite{GJ1979}).
\end{proof}

\section{Proof of Theorem~\ref{thm:main}}

There is only one difficult direction in our proof of Theorem~\ref{thm:main}. The following easy lemma handles  the ``only if" direction.

\begin{lem}\label{lem:only if}
Suppose $\gam$ is a simplicial graph on $n$ vertices which is $k$--colorable.
Then $A(\gam)$ admits a surjection as in Theorem~\ref{thm:main}.
\end{lem}
\begin{proof}
Let $\kappa$ be a $k$--coloring, and write $V_i=\kappa^{-1}(i)\subset V$. If $\{v,w\}\in E$ then $v\in V_i$ and $w\in V_j$ for suitable
indices $i\neq j$. We may thus form a quotient $G$ of $A(\gam)$ by imposing the relation $[a,b]=1$ for all pairs $a\in V_i$ and $b\in V_j$
for $i\neq j$. Clearly the result will be a right-angled Artin group $A(\Lambda)$, where \[\Lambda=V_1*V_2*\cdots *V_k\] is the join
of the sets $\{V_1,\ldots,V_k\}$. Since for all $i$ the set $V_i$ is totally disconnected in $\gam$ (and hence in $\Lambda$), it is easy
to see that $A(\Lambda)$ is a direct product of free groups with ranks $\{|V_1|,\ldots,|V_k|\}$, and that
\[\sum_{i=1}^k |V_i|=n,\] as desired.
\end{proof}

We now turn our attention to the ``if" direction. Suppose
\[\phi\colon A(\gam)\to\prod_{i=1}^k F_i=G\] is a surjection as in Theorem~\ref{thm:main}.
We will fix notation and write $\{v_1,\ldots,v_n\}$ for the vertices of $\gam$ and $w_i=\phi(v_i)$. Recall that $m_i$ stands for the rank of $F_i$. We write $X=\{x_1,\ldots,x_n\}$ for a
generating set of $G$, where \[X_i=\{x_{1+\sum_{j<i}m_j},x_{2+\sum_{j<i}m_j},\ldots,x_{\sum_{j\leq i} m_j}\}\] generates the subgroup
\[\{1\}\times\cdots\times\{1\}\times F_i\times\{1\}\times\cdots\times\{1\}.\] For each $x_i\in X$ and $g\in G$, we write $\exp_{x_i}(g)$ for the
\emph{exponent sum} of $x_i$ in $g$, i.e. the image of the element $g$ under the homomorphism $G\to\Z$ which sends $x_i$ to $1$ and
$x_j$ to $0$ for $j\neq i$.

If $g\in G$ is arbitrary, we write $g=g_1\cdots g_k$, where $g_i\in\langle X_i\rangle$. It is easy to see that this expression for $g$ is unique.

\begin{lem}\label{lem:dimension}
Suppose $g,h\in G$ are elements such that $[g,h]=1$. Write $g=g_1\cdots g_k$ and $h=h_1\cdots h_k$.
Then for all $1\leq i\leq k$, the tuples \[(\exp_{\alpha}(g_i))_{\alpha\in X_i},(\exp_{\alpha}(h_i))_{\alpha\in X_i}\]
are rational multiples of each other.
\end{lem}

Note that the map \[g_i\mapsto (\exp_{\alpha}(g_i))_{\alpha\in X_i}\] just computes the image of $g_i$ in the abelianization
$H_1(F_i,\Z)$, using the images of elements of $X_i$ as an additive basis for $\Z^{m_i}$.
\begin{proof}[Proof of Lemma~\ref{lem:dimension}]
Fix $i$ arbitrarily. If one of these tuples consists of all zeros then there is nothing to show. Otherwise, we may suppose that these tuples
are nontrivial for both $g_i$ and $h_i$. In particular, we must have that both $g_i$ and $h_i$ are nontrivial group elements of
$\langle X_i\rangle$. The
centralizer of a nontrivial element of a free group is cyclic, so that then $g_i$ and $h_i$ must share a common nonzero power. Since the
exponent sum map is a homomorphism, the lemma is now immediate.
\end{proof}

We will  require the  following fact from linear algebra, which  we include for completeness.

\begin{lem}\label{lem:lin-alg}
Let $M$ be a complex $n\times n$ matrix of rank $n$, let $1\leq k\leq n-1$ be an integer,
and consider the block decomposition $M=(M_1\mid M_2)$, where
$M_1$ is an $n\times k$ matrix. There exist $k$ rows $\{r_1,\ldots,r_k\}$ which have the following properties:
\begin{enumerate}
\item
The submatrix $M_1'$ of $M_1$ spanned by the rows $\{r_1,\ldots,r_k\}$ has rank $k$.
\item
The submatrix $M_2'$ of $M_2$ obtained by deleting the rows $\{r_1,\ldots,r_k\}$ has rank $n-k$.
\end{enumerate}
\end{lem}
\begin{proof}
We have that the determinant of $M$ is nonzero, since $M$ is invertible.
We may  now expand the determinant about $k\times k$--subminors of $M_1$. That is, we consider a submatrix $M_1'$ of $M_1$
spanned by $k$ rows, and the submatrix $M_2'$ of $M_2$ obtained by deleting the $k$ rows used to define $M_1'$. We then consider the
complex number $\zeta=(\det M_1')\cdot (\det M_2')$. An easy application of the Leibniz formula shows that $\det M$ is an alternating sum
of the complex numbers $\zeta$, as $M_1'$ ranges over all possible choices of $k$ rows. Since  $\det M\neq 0$, there is a choice of $M_1'$
and $M_2'$ so that the corresponding value of $\zeta$ is nonzero. In particular, such a choice of $M_1'$ and $M_2'$ gives the desired
matrices of the correct  ranks, whence the lemma follows.
\end{proof}

The following lemma completes the proof of Theorem~\ref{thm:main}.

\begin{lem}\label{lem:if}
Let $\phi\colon A(\gam)\to G$ be a surjective homomorphism as above. Then $\gam$ admits a $k$--coloring.
\end{lem}
\begin{proof}
Since the abelianizations of $A(\gam)$ and $G$ are both isomorphic to $\Z^n$, we have that the induced map
\[\phi_*\colon H_1(A(\gam),\Q)\rightarrow  H_1(G,\Q)\] is an isomorphism. We express this map as a matrix with respect to the additive bases
$\{v_1,\ldots,v_n\}$ for $H_1(A(\gam),\Q)$ and $\{x_1,\ldots, x_n\}$ for $H_1(G,\Q)$, so that \[\phi_*(v_i)=[w_i]=\sum_{j=1}^n \beta_j^i x_j\]
for suitable coefficients $\beta_j^i$. With respect to these bases, $\phi_*$ is represented by an $n\times n$ matrix
$A=(\beta_j^i)$ with rank exactly $n$, where here $\beta^i_j$ denotes the $(i,j)$--entry of $A$.
We write $A$ as a block column matrix \[A=(A_1\mid\cdots \mid A_k),\] where the column space of $A_i$ has dimension exactly $m_i$.

Consider the matrix $A_1$. Since $\phi_*$ is an isomorphism, we have that the row space of $A_1$ has dimension $m_1$. We may therefore
choose an $m_1\times m_1$ minor $B_1$ of $A_1$ with rank $m_1$. By
Lemma~\ref{lem:lin-alg}, we may assume that deleting the first $m_1$ columns of $A$ and the rows which appear in
$B_1$ results in a matrix of rank exactly $n-m_1$. We perform a row permutation of $A$
to get a matrix $A'$, so that the rows of $B_1$ are the
restriction of the first $m_1$ rows of $A'$ to the first $m_1$ columns of $A'$.

Deleting the first $m_1$ rows and columns of $A'$ results in a matrix of rank exactly $n-m_1$. Repeating this process for each block
column matrix results in an $n\times n$ block matrix \[B=(B_1\mid\cdots\mid B_k)\] with the following properties:
\begin{enumerate}
\item
The matrix $B$ is obtained from $A$  by permuting rows.
\item
The rows with index set $J_i=\{(\sum_{j<i}m_j)+1,\ldots,\sum_{j\leq i}m_j\}$ of $B_i$ have rank $m_i$.
\end{enumerate}

Observe that the rows of $B$ correspond to the group elements \[\{w_i=\phi(v_i)\}_{i=1}^n,\] where we have relabeled the
elements $\{v_1,\ldots,v_n\}$ so that $\phi(v_i)$ corresponds to the $i^{th}$ row of $B$. We define
$\kappa\colon V\to \{1,\ldots,k\}$ by
$\kappa(v_{\ell})=i$ if $\ell\in J_i$. It remains  to check that $\kappa$ is a valid coloring  of $\gam$. Suppose that $\kappa(v)=\kappa(w)=i$ for
distinct vertices of $V$. Then the elements $[\phi(v)]$ and $[\phi(w)]$ in $H_1(G,\Q)$ correspond to linearly independent rows of the block
$B_i$. We have that if $\{v,w\}\in E$ then $v$ and $w$ commute. Lemma~\ref{lem:dimension} implies that the rows corresponding to
$[\phi(v)]$ and $[\phi(w)]$ in the block $B_i$ are rational multiples of each other, which is a contradiction. We conclude that
$\{v,w\}\notin E$, so that $\kappa$ is a valid coloring of $\gam$.
\end{proof}

Observe that the proof of Theorem~\ref{thm:main} builds an explicit coloring of $\Gamma$ using the image of $V(\gam)$ under $\phi$.
In  principle,  an  algorithm which takes as  input the surjection from Theorem~\ref{thm:main} and outputs a coloring of $\gam$ can  be
implemented,  and even in polynomial time.

More precisely, let $\gam$ be a graph with $n$ vertices and let  $\phi$ be a surjection
from $A(\gam)$ to a product of  $k$ free  groups of total  rank $n$, given in terms of generators. We can compute the map induced by
$\phi$ on abelianizations in time which is linear  in the complexity of $\phi$.
Let $M$  be the resulting matrix  and $N$  a bound on the maximum of the entries of $M$, in absolute value.
Following the proof of Theorem~\ref{thm:main}, we break $M$ into  submatrices $M_1$  and $M_2$, and compute a sequence of
 determinants of submatrices of $M_1$ and $M_2$, of which there are at  most polynomially many
 as a function of $n$. All this would require
a  computation time which  is bounded by  a polynomial in $n$ and $N$. This allows us to sort the rows  of $M$ as in the proof of the
theorem and therefore produce a $k$--coloring in polynomial computing time.

It it straightforward to see  that  one  can start with a  $k$--coloring of $\gam$ and produce the surjection $\phi$ in polynomial time.

\section*{Acknowledgements}

\vspace{.5cm}

Ram\'{o}n Flores is supported by FEDER-MEC grant MTM2016-76453-C2-1-P and FEDER grant US-1263032 from the Andalusian
Government. Thomas Koberda is partially supported  by an
Alfred P. Sloan Foundation Research Fellowship, by NSF Grant DMS-1711488, and by NSF Grant DMS-2002596.
Delaram Kahrobaei is supported in part by a 
Canada's New Frontiers in Research Fund, under the Exploration grant entitled ``Algebraic Techniques for Quantum Security".
We thank Yago Antol\'{i}n, Juan Gonz\'alez-Meneses, Sang-hyun Kim,
and Andrew Sale for helpful discussions  and comments.
We thank the University of York for hospitality while part of this research
 was conducted. Finally, we thank an anonymous referee for helpful comments.
\bibliographystyle{amsplain}
\bibliography{ref}

\def\cprime{$'$} \def\soft#1{\leavevmode\setbox0=\hbox{h}\dimen7=\ht0\advance
  \dimen7 by-1ex\relax\if t#1\relax\rlap{\raise.6\dimen7
  \hbox{\kern.3ex\char'47}}#1\relax\else\if T#1\relax
  \rlap{\raise.5\dimen7\hbox{\kern1.3ex\char'47}}#1\relax \else\if
  d#1\relax\rlap{\raise.5\dimen7\hbox{\kern.9ex \char'47}}#1\relax\else\if
  D#1\relax\rlap{\raise.5\dimen7 \hbox{\kern1.4ex\char'47}}#1\relax\else\if
  l#1\relax \rlap{\raise.5\dimen7\hbox{\kern.4ex\char'47}}#1\relax \else\if
  L#1\relax\rlap{\raise.5\dimen7\hbox{\kern.7ex
  \char'47}}#1\relax\else\message{accent \string\soft \space #1 not
  defined!}#1\relax\fi\fi\fi\fi\fi\fi}
\providecommand{\bysame}{\leavevmode\hbox to3em{\hrulefill}\thinspace}
\providecommand{\MR}{\relax\ifhmode\unskip\space\fi MR }
\providecommand{\MRhref}[2]{%
  \href{http://www.ams.org/mathscinet-getitem?mr=#1}{#2}
}
\providecommand{\href}[2]{#2}
\begin{thebibliography}{10}

\bibitem{AB09}
Sanjeev Arora and Boaz Barak, \emph{Computational complexity}, Cambridge
  University Press, Cambridge, 2009, A modern approach. \MR{2500087}

\bibitem{BradyMeier01}
Noel Brady and John Meier, \emph{Connectivity at infinity for right angled
  {A}rtin groups}, Trans. Amer. Math. Soc. \textbf{353} (2001), no.~1,
  117--132. \MR{1675166}

\bibitem{CharneyGD}
Ruth Charney, \emph{An introduction to right-angled {A}rtin groups}, Geom.
  Dedicata \textbf{125} (2007), 141--158. \MR{2322545}

\bibitem{Droms87}
Carl Droms, \emph{Isomorphisms of graph groups}, Proc. Amer. Math. Soc.
  \textbf{100} (1987), no.~3, 407--408. \MR{891135}

\bibitem{FKK2019}
Ram\'{o}n Flores, Delaram Kahrobaei, and Thomas Koberda, \emph{Algorithmic
  problems in right-angled {A}rtin groups: complexity and applications}, J.
  Algebra \textbf{519} (2019), 111--129. \MR{3874519}

\bibitem{FKK2020exp}
{Flores, Ram\'{o}n and Kahrobaei, Delaram and Koberda, Thomas}, \emph{Expanders
  and right-angled {A}rtin groups}, Preprint (2020).

\bibitem{GJ1979}
Michael~R. Garey and David~S. Johnson, \emph{Computers and intractability}, W.
  H. Freeman and Co., San Francisco, Calif., 1979, A guide to the theory of
  NP-completeness, A Series of Books in the Mathematical Sciences. \MR{519066}

\bibitem{GMW91}
Oded Goldreich, Silvio Micali, and Avi Wigderson, \emph{Proofs that yield
  nothing but their validity, or {A}ll languages in {NP} have zero-knowledge
  proof systems}, J. Assoc. Comput. Mach. \textbf{38} (1991), no.~3, 691--729.
  \MR{1125927}

\bibitem{Gromov1987}
M.~Gromov, \emph{Hyperbolic groups}, Essays in group theory, Math. Sci. Res.
  Inst. Publ., vol.~8, Springer, New York, 1987, pp.~75--263.

\bibitem{HS2007}
Susan Hermiller and Zoran \v{S}uni\'{c}, \emph{Poly-free constructions for
  right-angled artin groups}, J. Group Theory \textbf{10} (2007), 117--138.
  \MR{2288463}

\bibitem{Kambites09}
Mark Kambites, \emph{On commuting elements and embeddings of graph groups and
  monoids}, Proc. Edinb. Math. Soc. (2) \textbf{52} (2009), no.~1, 155--170.
  \MR{2475886}

\bibitem{KK2013gt}
Sang-hyun Kim and Thomas Koberda, \emph{Embedability between right-angled
  {A}rtin groups}, Geom. Topol. \textbf{17} (2013), no.~1, 493--530.
  \MR{3039768}

\bibitem{KK2018jt}
Sang-Hyun Kim and Thomas Koberda, \emph{Free products and the algebraic
  structure of diffeomorphism groups}, J. Topol. \textbf{11} (2018), no.~4,
  1054--1076. \MR{3989437}

\bibitem{KoberdaGAFA}
Thomas Koberda, \emph{Right-angled {A}rtin groups and a generalized isomorphism
  problem for finitely generated subgroups of mapping class groups}, Geom.
  Funct. Anal. \textbf{22} (2012), no.~6, 1541--1590. \MR{3000498}

\bibitem{Minsky67}
Marvin~L. Minsky, \emph{Computation: finite and infinite machines},
  Prentice-Hall, Inc., Englewood Cliffs, N.J., 1967, Prentice-Hall Series in
  Automatic Computation. \MR{0356580}

\bibitem{Sabalka09}
Lucas Sabalka, \emph{On rigidity and the isomorphism problem for tree braid
  groups}, Groups Geom. Dyn. \textbf{3} (2009), no.~3, 469--523. \MR{2516176}

\bibitem{LMPT}
Massimiliano Sala, Teo Mora, Ludovic Perret, Shojiro Sakata, and Carlo Traverso
  (eds.), \emph{Gr\"{o}bner bases, coding, and cryptography}, Springer-Verlag,
  Berlin, 2009. \MR{2590633}

\bibitem{Servatius1989}
H.~Servatius, \emph{Automorphisms of graph groups}, J. Algebra \textbf{126}
  (1989), no.~1, 34--60. \MR{1023285 (90m:20043)}

\end{thebibliography}

\end{document}